\documentclass[11pt, reqno]{amsart}
\usepackage{amssymb, amsthm, amsmath, amsfonts}
\usepackage{graphics}
\usepackage{hyperref}
\usepackage[all]{xy}
\usepackage{enumerate}
\usepackage[mathscr]{eucal}
\usepackage{enumerate}
\usepackage{thmtools}

\declaretheorem{theorem}
\declaretheorem[name=Lemma, sibling=theorem]{lemma}
\declaretheorem[name=Proposition, sibling=theorem]{proposition}
\declaretheorem[name=Corollary, sibling=theorem]{corollary}
\declaretheorem[name=Definition, style=definition, sibling=theorem]{definition}

\newtheorem*{remark}{Remark}
\declaretheorem[name=Example, style=definition, sibling=theorem]{example}
\declaretheorem[name=Question, style=definition, sibling=theorem]{question}

\DeclareMathOperator{\reg}{reg}
\DeclareMathOperator{\Mod}{-Mod}

\DeclareMathOperator{\Ker}{Ker}
\DeclareMathOperator{\coKer}{coKer}
\DeclareMathOperator{\Image}{Im}
\DeclareMathOperator{\prd}{prd}
\DeclareMathOperator{\std}{std}
\DeclareMathOperator{\pd}{pd}

\newcommand{\C}{{\mathscr{C}}}
\newcommand{\D}{{\mathscr{D}}}
\newcommand{\h}{{\mathrm{H}}}
\newcommand{\N}{{\mathbb{N}}}
\newcommand{\kk}{{\mathbf{k}}}
\newcommand{\OI}{{\mathbf{OI}}}
\newcommand{\OB}{{\mathbf{OB}}}
\newcommand{\FI}{{\mathbf{FI}}}
\newcommand{\VI}{{\mathbf{VI}}}

\title{An inductive method for OI-modules}

\author{Wee Liang Gan}
\address{Department of Mathematics, University of California, Riverside, CA 92521, USA}
\email{wlgan@ucr.edu}

\author{Liping Li}
\address{LCSM (Ministry of Education), School of Mathematics and Statistics, Hunan Normal University, Changsha, Hunan 410081, China}
\email{lipingli@hunnu.edu.cn}

\thanks{L. Li s supported by the National Natural Science Foundation of China (Grant No. 11771135), the HuXiang High-Level Talents Gathering Project of Hunan Provincial Science and Technology Department (Grant No. 2019RS1039), and the Research Foundation of Hunan Provincial Education Department (Grant No. 18A016).
}

\begin{document}

\begin{abstract}
In this paper we introduce an inductive method to study $\OI$-modules presented in finite degrees, where $\OI$ is a skeleton of the category of totally ordered finite sets and order-preserving injective maps. As an application, we obtain an explicit upper bound for the Castelnuovo-Mumford regularity of $\OI$-modules.
\end{abstract}

\maketitle

\section{Introduction}

\subsection{Motivation} Let $F$ be a covariant (resp., contravariant) functor from a combinatorial category $\C$ to another category $\D$. In the situation that objects in $\D$ are equipped with a homology (resp., cohomology) theory (for instances, $\D$ are the categories of topological spaces, manifolds, algebras, groups, etc.), the composite of $F$ and the homology functors $\h_{\bullet}(-; R): \D \to R \Mod$ over a coefficient ring $R$ is a representation of $\C$, and it can be used to simultaneously explore the homology groups of a collection of objects in $\D$ parameterized by the object set of $\C$. This strategy recently forms the central theme of representation stability theory introduced by Church and Farb in \cite{cf}. They and quite a few authors have systematically studied the representation theoretic properties of the category $\FI$ of finite sets and injections, and applied them to explore stability patterns of (co)homological groups of many interesting examples such as configuration spaces, congruence subgroups, mapping class groups, etc; see for instances \cite{ce, cef, cefn, cmnr, gan, gl, gl2, lr, ly, mw, ramos}.

Another natural combinatorial category appearing in representation stability theory is the category $\OI$, whose objects are $[n]$ for $n \in \N$, and morphisms are strictly increasing maps. The category $\OI$ is closely related to  semisimplicial category (or called category $\Delta_+$ in literature) of finite totally ordered sets and strictly increasing maps, which is familiar to topologists as it has been used to define semisimplicial objects. Recently, some authors begin to consider representation theory of category $\OI$ and its applications in the study of homology of groups, and establish the following results: every finitely generated $\OI$-module over a commutative Noetherian ring is Noetherian, and its Castelnuovo-Mumford regularity is finite; the Hilbert function of a finitely generated $\OI$-module over a field is eventually polynomial; see for instances \cite{gl, pss, ss}. These results can be deduced from an inductive method introduced by the authors in \cite{gl, gl2}. However, compared to the fruitfulness of representation theory of $\FI$, many aspects of the structures of $\OI$-modules are still mysterious. In particular, as far as the authors know, quantitative results about $\OI$-modules such as upper bounds of regularity are still missing, which, as have been shown in the representation theory of $\FI$, are essential to bound stable ranges of (co)homology groups; see \cite{ce, cmnr, gl3, mw}.

The main goal of this paper is to introduce another inductive method for $\OI$-modules with two obvious advantages: it works for arbitrary $\OI$-modules rather than finitely generated $\OI$-modules over commutative Noetherian rings; and it can deduce some quantitative results such as explicit upper bounds of Castelnuovo-Mumford regularity and the natural number from which the eventually polynomial growth property of Hilbert functions starts. This inductive method is based on a key combinatorial proposition (see Proposition \ref{key proposition}) described in Section 3. Although similar results and proofs have been figured out by the authors for $\FI$-modules and $\VI$-modules (where $\VI$ is the category of finite dimensional vector spaces over a finite field and linear injections) in \cite{gl3, li}, we point out that the combinatorial structure of $\OI$ seems to be more complicated because of the lack of transitivity; that is, the endomorphism group of objects $x$ in $\OI$ are all trivial, and hence do not act transitively on the morphism sets ending at $x$. Therefore, the proof of this combinatorial proposition is very delicate, and the upper bounds we obtained in this paper are much larger compared to the upper bounds for $\FI$-modules (see \cite[Theorem A]{ce} and \cite[Theorem 1.3]{li}) and $\VI$-modules (see \cite[Theorem 1.1]{gl3}).

\subsection{Notations}
In this paper we let $\N$ be the set of non-negative integers. For any $n\in \N$, denote by $[n]$ the set $\{1, \, \ldots, \, n\}$; in particular, $[0]=\emptyset$ by convention. A map $\alpha: [m]\to [n]$ is increasing if $\alpha(1)<\cdots<\alpha(m)$. Let $\OI$ be the category whose objects are $[n]$ where $n=0, 1, \ldots$, and whose morphisms are the increasing maps. Note that $\OI$ is equivalent to the category of totally ordered finite sets and order-preserving injective maps.

Fix a commutative ring $\kk$. By an $\OI$-module, we mean a covariant functor from $\OI$ to the category of $\kk$-modules. Denote by $\OI\Mod$ the category of $\OI$-modules. This is an abelian category with enough projective objects. In particular, $\OI$-modules of the form $\kk \OI([n], -)$ (denoted by $M(n)$ later) are projective. For an $\OI$-module $V$ and $n\in \N$, we write $V_n$ for $V([n])$. If $V$ is nonzero, its degree $\deg V$ is defined to be $\sup \{ n \mid V_n \neq 0 \}$; otherwise, we set its degree to be $-1$.

For any $d\in \N$, we write $V_{\prec d}$ for the smallest $\OI$-submodule of $V$ containing $V_n$ for all $n<d$. In other words, $V_{\prec d}$ is the submodule of $V$ generated by all $V_n$'s with $n < d$. Define an $\OI$-submodule $U$ of $V$ by
\[ U_d = (V_{\prec d})_d \quad \mbox{ for every } d.\]
Define a functor $\h^{\OI}_0: \OI\Mod \to \OI\Mod$ by
\[ \h^{\OI}_0(V) = V/U. \]
The functor $\h^{\OI}_0$ is right-exact; let $\h^{\OI}_i$ be its $i$-th left derived functor, and set
\[ t_i(V) = \deg \h^{\OI}_i(V). \]
We call $t_0(V)$ the generation degree of $V$, $t_1(V)$ the relation degree, and $\prd{V} = \max\{t_0(V), \, t_1(V)\}$ the presentation degree of $V$. We say that $V$ is generated in finite degrees if $t_0(V)$ is finite, and $V$ is presented in finite degrees if the presentation degree of $V$ finite. The regularity $\reg(V)$ of $V$ is defined by
\[ \reg(V) = \sup\{ t_i(V)-i \mid i\in \N \}. \]

\begin{remark} \normalfont
In literature $\h^{\OI}_i(V)$ is called the $i$-th homology group of $V$, and the functor $\h^{\OI}_i$ is interpreted by the more traditional Tor functor using the notion of category algebras. In this paper we do not take this approach. For details, please refer to \cite{gl}.
\end{remark}

\subsection{Self-embedding and shift functor.} We now define a self-embedding functor $\sigma: \OI \to \OI$ as follows. For each object $[n]$ of $\OI$, let
\[ \sigma([n])=[n+1]. \]
For each morphism $\alpha: [m]\to [n]$ of $\OI$, define $\sigma(\alpha): [m+1] \to [n+1]$ by
\[ \sigma(\alpha)(h) = \left\{ \begin{array}{ll}
1 & \mbox{ if } h=1,\\
\alpha(h-1)+1 & \mbox{ if } h>1.
\end{array}\right. \]
The functor $\sigma$ induces a shift functor $\Sigma: \OI\Mod \to \OI\Mod$ by defining
\[ \Sigma V = V\circ\sigma \quad \mbox{ for every } V\in \OI\Mod. \]
Note that for every $n\in\N$, one has: $(\Sigma V)_n = V_{n+1}$. For every $r\in \N$, we write $\Sigma^r$ for $\underbrace{\Sigma \circ \cdots \circ \Sigma}_r$.

For each $n\in \N$, we write $\iota : [n] \to [n+1]$ for the morphism of $\OI$ defined by
\[ \iota(h) = h+1 \quad \mbox{ for every } h\in [n]. \]
For any $\OI$-module $V$, there is a natural $\OI$-module homomorphism $V\to \Sigma V$ defined by
\[ V_n \to (\Sigma V)_n, \quad v \mapsto \iota v, \quad \mbox{ for every } n\in \N. \]
Let $\kappa V$ and $\Delta V$ be, respectively, the kernel and cokernel of $V\to \Sigma V$.

\subsection{Main results}

We now state our main results. The first theorem, though it might seem technical, actually lays the foundation for us to develop an inductive method similar to that described in \cite{gl}.

\begin{theorem} \label{shift theorem}
Let $V$ be an $\OI$-module presented in finite degrees, $d= t_0(V)$, and $r$ be any integer such that $r \geqslant \prd(V)$. Define
\[ \overline{V} = \frac{\Sigma^r V}{(\Sigma^r V)_{\prec d}}. \]
Then $\kappa \overline{V} = 0$.
\end{theorem}

Based on the above result, we develop a formal inductive method which allows us to verify representation theoretic properties of $\OI$-modules presented in finite degrees in a convenient way. We refer the reader to Definition \ref{properties of properties} for precise meanings of terminologies in the theorem.

\begin{theorem} \label{inductive machinery}
Let (P) be a property of some $\OI$-modules and suppose that the zero module has property (P). Then every $\OI$-module presented in finite degrees has property (P) if and only if (P) is glueable, $\Sigma$-dominant, and $\Delta$-predominant.
\end{theorem}

We point out that this theorem is definitely not a superficial extension of \cite[Theorem 1.8]{gl} from the category of finitely generated $\OI$-modules to the category of $\OI$-modules presented in finite degrees. Actually, the former one relies heavily on the Noetherian property of finitely generated modules over commutative Noetherian rings, while the second one is based on a completely novel machinery working for all $\OI$-modules.

If we apply $\Sigma^n$ to an $\OI$-module $V$ presented in finite degrees, it may not become a semi-induced module for $n \gg 0$. This is a big difference between $\FI$-modules and $\OI$-modules. However, we can still get a weaker stability phenomenon called filtration stability.

\begin{theorem} \label{filtration stability}
Let $V$ be an $\OI$-module presented in finite degrees. Then there exist a finite collection of $\OI$-modules $\mathscr{F}_V = \{V^1, \ldots, V^s \}$ and an integer $N \in \N$ such that for every $n \geqslant N$, there is a finite filtration on $\Sigma^n V$ with the property that each successive quotient is isomorphic to a member $V^i \in \mathscr{F}_V$. Moreover, $t_0(V^i) \leqslant t_0(V)$ for $i \in [s]$ and $s \leqslant 2^{t_0(V) + 1} - 1$.
\end{theorem}

It has been shown by the authors in \cite{gl} that a finitely generated $\OI$-module over a commutative Noetherian ring has finite regularity. Theorem \ref{inductive machinery} allows us to establish the finiteness of regularity of $\OI$-modules presented in finite degrees. Furthermore, combining the quantitative result in Theorem \ref{shift theorem}, we can obtain a simple upper bound of regularity for all $\OI$-modules.

\begin{theorem} \label{regularity theorem}
For any nonzero $\OI$-module $V$, one has:
\[ \reg(V) \leqslant 2^{2^{t_0(V)}} \prd(V). \]
\end{theorem}

Note that if $V$ is the zero $\OI$-module, then $\reg(V)=-1$. Thus in the above theorem we require $V$ to be nonzero. We also point out that the theorem holds trivially if $t_0(V)$ or $t_1(V)$ is infinite. As an immediate corollary, we deduce that the category of $\OI$-modules presented in finite degrees is an abelian subcategory of $\OI \Mod$, so we can do homological algebra safely in it.

It was proved in \cite[Corollary 1.13]{gl} that for a finitely generated $\OI$-module $V$ over a field, there exists a positive integer $N_V$ such that its Hilbert function eventually coincides with a rational polynomial for $n \geqslant N_V$. The following theorem provides a bound for $N_V$.

\begin{theorem}\label{hilbert functions}
Let $V$ be a finitely generated $\OI$-module over a field $\kk$. Then there exists a rational polynomial $P$ such that $\dim_{\kk} V_n = P(n)$ whenever
\[ n \geqslant 2^{2^{t_0(V)}} \prd(V). \]
Moreover, the degree of $P$ is at most $t_0(V)$.
\end{theorem}

The next theorem classifies all $\OI$-modules presented in finite degrees which are acyclic with respect to the homology functors, a result analogue to \cite[Theorem 1.3]{ly} of $\FI$-modules.

\begin{theorem} \label{homologically acyclic modules}
Let $V$ be an $\OI$-module presented in finite degrees. Then the following statements are equivalent:
\begin{enumerate}
\item $V$ is semi-induced;
\item $\h^{\OI}_i(V) = 0$ for all $i \geqslant 1$;
\item $\h^{\OI}_i(V)$ for a certain $i \in \mathbb{N}$.
\end{enumerate}
\end{theorem}

For an definition of semi-induced modules, please refer to Subsection 4.6.

\subsection{Applications in other areas} Recently, G\"{u}nt\"{u}rk\"{u}n and Snowden gave a very comprehensive picture of the representation theory of the increasing monoid over a field $\kk$ in \cite{gs}. They noted that the module category of the increasing monoind is essentially equivalent to the module category of $\OI$, the module category of the semisimplical category, and the module category of the shuffle algebra freely generated by one element in degree 1; see \cite[Remark 3.3, Propositions 3.7, 3.8, 3.10]{gs}. Since these equivalences are actually independent of the coefficient ring $\kk$, our work can be applied to the above mentioned structures for any commutative ring. In particular, some results included in this paper can give us a much better understanding on representations of these structures. For instance, it is conjectured in \cite[Subsubsection 1.4.2]{gs} that if an $\OI$-module $V$ has level at most $r$, then there exists a positive integer $n(r)$ such that $\reg(V)$ is actually bounded in terms of $t_0(V), \, \ldots \, t_{n(r)}(V)$. Theorem \ref{regularity theorem} shows that the $\reg(V)$ of any $\OI$-module $V$ is bounded in terms of only $t_0(V)$ and $t_1(V)$, and moreover, an explicit simple formula in terms of them is provided. As a consequence, when investigating homological properties of $\OI$-modules over a non-Noetherian coefficient ring, it is possible to work in the category of $\OI$-modules presented in finite degrees rather than the category of finitely generated $\OI$-modules.

\subsection{Organization} The paper is organized as follows. In Section 2 we describe some preliminary results about the category $\OI$ and its representations. In Section 3 we prove a key combinatorial proposition. Main theorems and some corollaries are proved in Section 4. In the last section we raise some questions which might be of certain interest to the reader.

\section{Preliminaries}

In this section we list some preliminary results on $\OI$-modules. These results have appeared in literature, and were proved for general categories (including $\OI$ as an example) equipped with shift functors satisfying certain axioms. For details, please refer to \cite{gl}.

For any $m\in \N$, let $M(m)$ be the $\OI$-module that takes each $[n]$ to the free $\kk$-module on the set of increasing maps from $[m]$ to $[n]$. If $\beta: [s]\to[n]$ is any morphism of $\OI$, the induced map $M(m)_s \to M(m)_n$ is defined by
\[  \sum_{\alpha: [m]\to [s]} c_\alpha \alpha \mapsto \sum_{\alpha: [m]\to [s]} c_\alpha \beta \alpha\]
where $c_\alpha$ are coefficients in $\kk$ and $\alpha$ runs over the increasing maps $[m]\to [s]$. Note that $t_0(M(m))=m$. It is easy to check that $M(m)$ is a projective $\OI$-module, so one has $t_i(M(m)) = -1$ for every $i\geqslant 1$. Furthermore, for any $\OI$-module $V$, there exists a surjective homomorphism $F\to V$ where $F=\bigoplus_{j\in J} M(d_j)$ for some $d_j\in \N$ such that $t_0(F)=t_0(V)$.

Let $m, r\in \N$. Let $E \subset [r]$ and suppose $|E|\leqslant m$. Write $E=\{e_1, \, \ldots, \, e_\ell\}$ where $e_1<\cdots<e_\ell$. For any increasing map $\alpha: [m-\ell] \to [n]$, define an increasing map
\[ \alpha_E: [m] \to [n+r] \]
by
\[ \alpha_E(h) = \left\{ \begin{array}{ll}
e_h & \mbox{ if } h\leqslant \ell,\\
\alpha(h-\ell)+r & \mbox{ if } h>\ell.
\end{array} \right. \]
Define an $\OI$-module homomorphism
\[ \theta_E: M(m-\ell)\to \Sigma^r M(m) \]
by
\[ M(m-\ell)_n\to (\Sigma^r M(m))_n, \quad \alpha \mapsto \alpha_E\]
for every $n\in \N$ and increasing map $\alpha: [m-\ell] \to [n]$. Note that if $r=1$ and $E=\emptyset$, then $\theta_E$ is the natural map $M(m)\to \Sigma M(m)$.

\begin{lemma} \label{decomposition lemma}
For any $m, r\in \N$, there is a natural isomorphism
\[ \theta: \bigoplus_{\ell=0}^{m} \bigoplus_{\substack{E\subset [r] \\ |E|=\ell} } M(m-\ell) \to \Sigma^r M(m). \]
\end{lemma}
\begin{proof}
Let $\theta$ be the homomorphism whose restriction to the direct summand $M(m-\ell)$ indexed by $E$ is $\theta_E$. It is easy to check that $\theta$ is an isomorphism.
\end{proof}

It follows from Lemma \ref{decomposition lemma} that for any $\OI$-module $V$ and $r\in \N$, one has:
\[ t_0(\Sigma^r V) \leqslant t_0(V); \]
if $V$ is nonzero, then one has:
\[ t_0(\Delta V) \leqslant t_0(V)-1, \]
and the equality of the second formula holds whenever $V$ is nonzero. The reader may refer to \cite[Lemma 2.2]{gl} for details. These inequalities also hold for regularity. That is:

\begin{lemma} \label{reg}
Let $V$ be an $\OI$-module. Then one has:
\[ \reg(V) \leqslant \reg(\Sigma V) + 1. \]
If $\kappa V=0$, then one has:
\[ \reg(V) \leqslant \reg(\Delta V) + 1. \]
\end{lemma}
\begin{proof}
See \cite[Corollary 5.3]{gl} and \cite[Corollary 5.6]{gl}.
\end{proof}

\section{A key proposition}

In this section we prove a combinatorial proposition, which plays the central role for the proof of Theorem \ref{shift theorem}. For this purpose we introduce a few notations.

Let $V$ be an $\OI$-module such that $\max\{t_0(V), \, t_1(V)\}<\infty$, $d= t_0(V)$, and let $r$ be any integer $\geqslant \max\{t_0(V), \, t_1(V)\}$. Then there exists a surjective homomorphism $F\to V$ where
\[ F=\bigoplus_{j\in J} M(d_j) \]
for some indexing set $J$ and $d_j\in \N$ such that $d_j \leqslant d$ for every $j\in J$. Let $W$ be the kernel of $F \to V$. Then
\[ t_0(W)\leqslant \max\{ t_0(V), \, t_1(V) \} \leqslant r. \]

Let
\[ I=\{ i\in J \mid d_i = d\}, \quad P=\bigoplus_{i\in I} M(d_i) = \bigoplus_{i\in I} M(d). \]
By Lemma \ref{decomposition lemma}, there is a natural decomposition
\[ \Sigma^r F \cong P \oplus Q \]
where $Q$ is a direct sum of $\OI$-modules of the form $M(m)$ such that $m<d$. Let
\[ \eta: \Sigma^r F \to P \]
be the projection with kernel $Q$, and
\[ \widehat{W} = \eta(\Sigma^r W) \subset P. \]

For a morphism $\alpha: [d] \to [s]$ in $\OI$, we let $\ell = \alpha(1)$, and define
\[ \widehat{\alpha} : [d] \to [s-\ell+1] , \qquad \widehat{\alpha}(h) = \alpha(h)-\ell+1 \mbox{ for } h=1,\ldots, d. \]
It is always true that $\widehat{\alpha}(1)=1$. Since every element in $F_s = \bigoplus_{j\in J} M(d_j)_s$ is a linear combination of morphisms starting at a certain object $[d_j]$ with $j \in J$ and ending at the object $[s]$, for $w \in W_s \subset F_s$, we can write
\begin{equation} \label{element of W}
w= \sum_{j\in J} \sum_{\alpha_j:[d_j]\to [s]} c_{j,\alpha} \alpha_j,
\end{equation}
where $c_{j,\alpha}$ are coefficients in $\kk$ and $\alpha_j$ runs over all morphisms from $[d_j]$ to $[s]$. If $j \in I$, then $d_j = d$ by definition, so in this case we simply write $\alpha$ and $\widehat{\alpha}$ rather than $\alpha_j$ and $\widehat{\alpha}_j$ in the above expression. For each $\ell = 1,\ldots, s-d+1$, we let
\begin{equation} \label{generators of W hat}
\widehat{w}_\ell = \sum_{i\in I}\sum_{\substack{\alpha:[d]\to [s] \\ \alpha(1)=\ell}} c_{i,\alpha} \widehat{\alpha} \in P_{s-\ell+1}.
\end{equation}
In particular, if $d>s$, then $\widehat{w}_\ell = 0$.

\begin{proposition} \label{key proposition}
The $\OI$-module $\widehat{W}$ is generated by the collection of elements $\widehat{w}_\ell$ for all $w\in W_s$ such that $s\leqslant t_0(W)$ and for all $\ell=1,\ldots, s-d+1$.
\end{proposition}

\begin{proof}
Take any $w\in W_s$ where $s\leqslant t_0(W)$ and write it in the form of \eqref{element of W}. For any $n\in \N$, we consider an increasing map $\beta: [s] \to [r+n]$. Then
\[ \beta w = \sum_{j\in J} \sum_{\alpha_j: [d_j]\to[s]}  c_{j,\alpha_j} \beta\alpha_j \in F_{r+n}. \]
Note that:
\begin{itemize}

\item
If $j\notin I$, then $\beta\alpha_j\in Q_n \subset (\Sigma^r F)_n = F_{r+n}$, so $\eta(\beta\alpha_j)=0$.

\item
If $j \in I$ and $\beta \alpha_j (1) \leqslant r$, then $\beta \alpha_j \in Q_n \subset (\Sigma^r F)_n = F_{r+n}$, so $\eta(\beta\alpha_j)=0$.

\end{itemize}
Therefore
\[ \eta(\beta w) = \sum_{i\in I} \sum_{\substack{\alpha: [d]\to [s] \\ \beta\alpha(1)> r }} c_{i,\alpha} \eta(\beta \alpha) \in P_n. \]

To show that $\widehat{w}_\ell \in \widehat{W}$ for each $\ell=1,\ldots, s-d+1$, we do a downward-induction on $\ell$. Let $\beta: [s] \to [s+r-\ell+1]$ be the map defined by
\[ \beta(h) = h+r-\ell+1  \quad \mbox{ for each } h. \]
So $\beta$ will map $P_s$ to $P_{s+r-\ell+1}$. Note that for any $\alpha:[d]\to [s]$, we have:
\[
\beta(\alpha(1)) = \alpha(1)+r-\ell + 1   \left\{ \begin{array}{ll}
\leqslant r & \mbox{ if } \alpha(1)<\ell,\\
= r+1 & \mbox{ if } \alpha(1)=\ell,\\
>r+1 &\mbox{ if } \alpha(1)>\ell.
\end{array}
\right.
\]
So
\begin{align*}
\eta(\beta w) &= \sum_{i\in I} \sum_{\substack{\alpha: [d]\to [s] \\ \alpha(1)\geqslant \ell }} c_{i,\alpha} \eta(\beta \alpha)    \\
&= \widehat{w}_\ell + \iota \widehat{w}_{\ell+1} + \iota^2 (\widehat{w}_{\ell+2}) + \cdots.
\end{align*}
By the downward-induction hypothesis, $\iota \widehat{w}_{\ell+1}, \iota^2 (\widehat{w}_{\ell+2}), \ldots$ are in $\widehat{W}$, so it follows that $\widehat{w}_\ell$ is also in $\widehat{W}$.

It remains to show that for \emph{any} increasing map $\beta: [s] \to [r+n]$, the element $\eta(\beta w)$ is contained in the submodule of $P$ generated by the collection of $\widehat{w}_\ell$. Note that
\begin{align*}
\eta(\beta w) &= \sum_{i\in I} \sum_{\substack{\alpha: [d]\to [s] \\ \beta\alpha(1)> r }} c_{i,\alpha} \eta(\beta \alpha) \\
&= \sum_{ \{\ell \,\mid \, \beta(\ell)>r\} } \sum_{i \in I } \sum_{\substack{\alpha: [d]\to [s] \\ \alpha(1)=\ell}} c_{i,\alpha} \eta(\beta \alpha) \in P_n.
\end{align*}
It suffices to check that for each $\ell$ such that $\beta(\ell)>r$, the element
\[ \sum_{i \in I } \sum_{\substack{\alpha: [d]\to [s] \\ \alpha(1)=\ell}} c_{i,\alpha} \eta(\beta \alpha) \]
is of the form $\gamma \widehat{w}_\ell$ for some $\gamma: [s-\ell+1]\to [n] $. To this end, define the map
\[ \gamma : [s-\ell+1] \to [n] \]
by
\[ \gamma(h) = \beta(h+\ell-1)-r \quad \mbox{ for each } h=1,\ldots, s-\ell+1. \]
Then for each $h=1,\ldots, d$, we have:
\[ \gamma \widehat{\alpha} (h)= \gamma(\alpha(h)-\ell+1) = \beta \alpha(h) -r; \]
this implies $\gamma \widehat{\alpha} = \eta(\beta \alpha)$. Hence,
\[ \gamma \widehat{w}_\ell = \sum_{i\ \in I} \sum_{\substack{\alpha: [d] \to [s] \\ \alpha(1)=\ell }} c_{i,\alpha} \gamma \widehat{\alpha}
= \sum_{i \in I } \sum_{\substack{\alpha: [d]\to [s] \\ \alpha(1)=\ell}} c_{i,\alpha} \eta(\beta \alpha),\]
as claimed.
\end{proof}

\section{Proofs of Main Theorems}

In this section we keep the notation of the preceding, and prove the main theorems stated in the introduction.

\subsection{A proof of Theorem \ref{shift theorem}}

In this subsection we prove the following theorem.

\begin{theorem}
Let $V$ be an $\OI$-module presented in finite degrees, $d= t_0(V)$, and $r$ be any integer such that $r \geqslant \prd(V)$. Define
\[ \overline{V} = \frac{\Sigma^r V}{(\Sigma^r V)_{\prec d}}. \]
Then $\kappa \overline{V} = 0$.
\end{theorem}

\begin{proof}
Since the functor $\Sigma^r$ is exact, we have the short exact sequence
\[  0 \to \Sigma^r W \to \Sigma^r F \to \Sigma^r V \to 0. \]
Recall the natural decomposition $\Sigma^r F \cong P \oplus Q$. The restriction of the map $\Sigma^r F \to \Sigma^r V$ to $Q$ gives a surjective homomorphism $Q \to (\Sigma^r V)_{\prec d}$. We get the following commuting diagram whose rows and columns are exact:
\[
\xymatrix{
0 \ar[r] & Q \ar[r] \ar[d] & \Sigma^r F \ar[r]^-{\eta} \ar[d] & P \ar[r] \ar[d] & 0 \\
0 \ar[r] & (\Sigma^r V)_{\prec d} \ar[r] \ar[d] & \Sigma^r V \ar[r] \ar[d] & \overline{V} \ar[r] \ar[d] & 0 \\
& 0 & 0 & 0 &
}
\]
By the snake lemma, the kernel of $P\rightarrow \overline{V}$ is $\widehat{W}$. Therefore $\overline{V}$ is isomorphic to $P / \widehat{W}$.

We now prove that $\kappa \overline{V} = 0$. Suppose $v\in P_n$ for some $n\in \N$ and $\iota v \in \widehat{W}_{n+1}$. We need to show that $v\in \widehat{W}_n$.

By Proposition \ref{key proposition}, we can write
\[ \iota v = \sum_{k=1}^p  a_k \beta_k \omega_k \in \widehat{W}_{n+1} \subset P_{n+1} \]
where $a_k$ are coefficients in $\kk$, each $\omega_k \in \widehat{W}_{s_k}$ is an element of the form:
\[ \sum_{i\in I} \sum_{\substack{\alpha : [d] \to [s_k]\\ \alpha(1)=1}}  c_{i,\alpha} \alpha  \in P_{s_k} \quad \mbox{ for some coefficients } c_{i,\alpha} \in \kk, \]
and $\beta_k : [s_k] \to [n+1]$. We write this as:
\[ \iota v = \sum_{ \{ k \,\mid \, \beta_k(1)=1 \} }  a_k \beta_k \omega_k + \sum_{ \{ k \, \mid \, \beta_k(1)>1 \} }  a_k \beta_k \omega_k. \]
Since $\iota(1)>1$, we must have:
\[ \iota v = \sum_{ \{ k \, \mid \, \beta_k(1)>1 \} }  a_k \beta_k \omega_k. \]
Now for each $\beta_k$ such that $\beta_k(1)>1$, define $\beta'_k: [s_k] \to [n]$ by $\beta'_k(h) = \beta_k(h)-1$ for each $h$, so that $\iota\beta'_k =\beta_k$.
Then
\[ \iota v = \iota\left( \sum_{ \{ k \, \mid \, \beta_k(1)>1 \} }  a_k \beta'_k \omega_k. \right). \]
But $\iota: P_n\to P_{n+1}$ is injective, so
\[ v = \sum_{ \{ k \, \mid \, \beta_k(1)>1 \} }  a_k \beta'_k \omega_k \in \widehat{W}_n. \]
This concludes the proof of Theorem \ref{shift theorem}.
\end{proof}

The importance of this theorem lies in the following fact. Inductive methods such as the one described in \cite{gl} have played a prominent role in representation stability theory. To apply these methods, in general the first step is to convert an arbitrary module $V$ into a closely related torsion free module $\overline{V}$ such that $\kappa \overline{V} = 0$. For finitely generated $\OI$-modules over commutative Noetherian rings, the authors have described a finite procedure to get such a $\overline{V}$ in \cite{gl}. However, the finiteness of this procedure highly depends on the Noetherian property of finitely generated modules, and hence it can not extend to the more general framework of arbitrary $\OI$-modules $V$ presented in finite degrees. The above theorem provides a redemption for this failure. To avoid confusion, we remind the reader that the module $V_{\mathrm{reg}}$ defined in \cite[Section 3]{gl} does not coincide with $\overline{V}$ in this paper. Here is an example:

\begin{example} \normalfont
Let $V = M(1) \oplus M(0)$. Then $t_0(V) = 1$ and $t_1(V) = -1$. For any $r \geqslant 1$,
\begin{equation*}
\Sigma^r V = M(1) \oplus M(0)^{\oplus r+1} = V \oplus M(0)^{\oplus r}.
\end{equation*}
Consequently,
\begin{equation*}
\overline{V} = \frac{\Sigma^r V}{(\Sigma^r V)_{\prec 1}} = \frac{M(1) \oplus M(0)^{\oplus r+1}} {M(0)^{\oplus r+1}} = M(1).
\end{equation*}
However, one has
\begin{equation*}
V_{\mathrm{reg}} = V = M(1) \oplus M(0).
\end{equation*}
\end{example}

For an $\FI$-module $V$ presented finite degrees, after applying the shift functor $\Sigma$ (for $\FI$-modules) $r$ times with $r \geqslant t_0(V) + t_1(V)$, one gets a semi-induced module, and in particular $\kappa \Sigma^r V = 0$; see for instances \cite[Theorem 2.6]{ns}, \cite[Corollary 3.3]{li}, or \cite[Theorem C]{ramos}. However, for $\OI$-modules, this result does not hold. The following example is provided by Eric Ramos:

\begin{example} \label{example}
Let $V$ be the following $\OI$-module: $V_0 = 0$, and $V_n = \kk$ for $n \geqslant 1$. For a morphism $\alpha: [m] \to [n]$ with $m \geqslant 1$, one defines $V(\alpha): V_m \to V_n$ to be the identity map if $\alpha(m) = n$ and $V(\alpha)$ is the zero map otherwise. The reader can check that $V$ is indeed an $\OI$-module with $t_0(V) = 1$ and $t_1(V) = 2$. A direct computation shows $\Sigma^r V = V \oplus U$ for any $r \geqslant 1$, where $U_0 = \kk$ and $U_n = 0$ for all $n \geqslant 1$, and $\overline{V} \cong V$. Therefore, $\Sigma^r V$ is not a semi-induced module for any $r \in \N$. But the natural map $\overline{V} \to \Sigma \overline{V} \cong U \oplus V$ is the obvious embedding map, and hence $\kappa \overline{V} = 0$.
\end{example}

\subsection{An upper bound of regularity}

In this subsection we use Theorem \ref{shift theorem} and an inductive method to establish Theorem \ref{regularity theorem}. This method is base on the following two short exact sequences:
\begin{align*}
& 0 \to (\Sigma^rV)_{\prec d} \to \Sigma^r V \to \overline{V} \to 0,\\
& 0 \to \overline{V} \to \Sigma \overline{V} \to \Delta \overline{V} \to 0.
\end{align*}

\begin{lemma} \label{inductive inequalities}
Let $r$ be an integer $\geqslant \max \{ t_0(V), \, t_1(V) \}$. We have:
\[ t_0(\Delta\overline{V}) \leqslant d-1 \mbox{ and } t_1(\Delta\overline{V}) \leqslant r-1.\]
\end{lemma}

\begin{proof}
To check the first inequality, one looks at the first short exact sequence and note that
\[t_0(\overline{V}) \leqslant t_0(\Sigma^r V) \leqslant t_0(V) = d,\]
so the conclusion holds; see the paragraph before Lemma \ref{reg}.

Now we turn to the second inequality. Applying the snake Lemma to the commutative diagram in the proof of Theorem \ref{shift theorem} we get a short exact sequence
\[
0 \to U \to \Sigma^r W \to \widehat{W} \to 0.
\]
Consequently, one gets
\[t_0(\widehat{W}) \leqslant t_0(\Sigma^r W)  \leqslant t_0(W) \leqslant r.\]
Furthermore, since $\Delta$ is right exact, applying it to the short exact sequence
\[ 0 \to \widehat{W} \to P \to \overline{V} \to 0  \]
we obtain another short exact sequence
\[ 0 \to C \to \Delta P \to \Delta \overline{V} \to 0 \]
where $C$ is a quotient module of $\Delta \widehat{W}$. Consequently, one has
\[ t_1(\Delta\overline{V}) \leqslant t_0(C) \leqslant t_0(\Delta \widehat{W}) \leqslant t_0(\widehat{W}) - 1 \leqslant r-1\]
as claimed.
\end{proof}

To prove the upper bound for regularity of $\OI$-modules presented in finite degrees, for each $d \in \mathbb{N}$ we introduce an auxiliary function $C_d: \mathbb{Z} \to \mathbb{Z}$ by the initial condition $C_0(r)=r$ and the recursive relation
\[ C_d(r) = C_{d-1}\left(C_{d-1}(r-1)+3\right) + r. \]
These functions are increasing with respect to $d$ and $r$, and have simple lower and upper bounds.

\begin{lemma} \label{increasing}
Suppose $r\geqslant d\geqslant 0$.
\begin{enumerate}[(a)]
\item
One has:
\begin{gather*}
C_d(r)\geqslant r,\\
C_d(r+1) > C_d(r).
\end{gather*}

\item
If $d\geqslant 1$, then
\[ C_d(r) > C_{d-1}(r). \]

\item Suppose $r\geqslant d\geqslant 0$. Then one has:
\[ C_d(r) \leqslant 2^{2^d} r. \]
\end{enumerate}
\end{lemma}

\begin{proof}
(a) We use induction on $d$. When $d=0$, the inequalities are obvious. Suppose $d\geqslant 1$. Then
\[ C_{d-1}(r-1) \geqslant r-1. \]
So we have:
\begin{align*}
C_d(r) &= C_{d-1}\left(C_{d-1}(r-1)+3\right) + r \\
&\geqslant C_{d-1}((r-1) + 3) + r \\
&\geqslant r-1 + 3 + r \\
&\geqslant r,
\end{align*}
and
\begin{align*}
C_d(r+1) &= C_{d-1}\left(C_{d-1}(r)+3\right) + r + 1 \\
&> C_{d-1}\left(C_{d-1}(r-1)+3\right) + r \\
&= C_d(r).
\end{align*}

(b) Using (a), we have:
\begin{align*}
C_d(r) &= C_{d-1}\left(C_{d-1}(r-1)+3\right) + r \\
&\geqslant C_{d-1} (r-1+3) + r\\
&> C_{d-1}(r).
\end{align*}

(c) Note that $C_0(r)=r$ and $C_1(r)=2r+2$. Let us prove that:
\[ C_d(r) \leqslant (2^{2^d}-1) r \quad \mbox{ for } r\geqslant d\geqslant 2. \]
We have:
\begin{align*}
C_2(r) &= C_1(C_1(r-1)+3) + r \\
&= 2( 2(r-1)+2 + 3  ) + 2 + r \\
&= 5r + 8 \\
&\leqslant 15 r.
\end{align*}
We use induction for $d>2$. By the induction hypothesis and the conclusion of Parts (a) and (b), we have:
\begin{align*}
C_d(r) &\leqslant (2^{2^{d-1}} - 1) ( ( 2^{2^{d-1}} -1 )(r-1) + 3) + r    \\
& = (2^{2^{d-1}} - 1) ( 2^{2^{d-1}}r - r - 2^{2^{d-1}} + 4 ) + r \\
& \leqslant (2^{2^{d-1}}) (2^{2^{d-1}}r - r ) + r \\
& = 2^{2^{d}} r - 2^{2^{d-1}} r + r \\
& \leqslant (2^{2^d} - 1 )r.
\end{align*}
\end{proof}

Now we are ready to prove Theorem \ref{regularity theorem} by an induction on $d = t_0(V)$,

\begin{theorem}
For any nonzero $\OI$-module $V$, one has:
\[ \reg(V) \leqslant   2^{2^{t_0(V)}} \prd(V). \]
\end{theorem}

\begin{proof}
Let $V$ be a nonzero $\OI$-module. If $\prd(V) = \infty$, the statement of Theorem \ref{regularity theorem} is trivial, so assume that $\prd(V) < \infty$. Let $d=t_0(V)$ and $r$ be an integer such that $r \geqslant \prd(V)$. We use an induction on $d$ to show that
\begin{equation} \label{bounding reg by C}
\reg(V)\leqslant C_d(r).
\end{equation}
This inequality clearly implies the conclusion of the theorem.

Suppose $d=0$.  Then
\[ t_0(\Delta\overline{V}) \leqslant -1 \quad\Rightarrow\quad \Delta\overline{V}=0 \quad\Rightarrow\quad \reg(\Delta\overline{V})=-1. \]
By Theorem \ref{shift theorem} and Lemma \ref{reg}, we have: $\reg(\overline{V})\leqslant 0$. Since $(\Sigma^r V)_{\prec 0}=0$, we have: $\Sigma^r V \cong \overline{V}$, so $\reg(\Sigma^r V) \leqslant 0$. By Lemma \ref{reg}, it follows that $\reg V \leqslant r = C_0(r)$.

Suppose $d\geqslant 1$. We shall use Lemma \ref{increasing} several times below without further mention. By the induction hypothesis, we have:
\[ \reg(\Delta\overline{V}) \leqslant C_{d-1}(r-1). \]
By Theorem \ref{shift theorem} and Lemma \ref{reg}, we have:
\[  \reg(\overline{V}) \leqslant C_{d-1}(r-1) + 1,
\]
so
\[ t_2(\overline{V}) \leqslant C_{d-1}(r-1) + 3. \]
Therefore, from the short exact sequence
\[ 0 \to (\Sigma^rV)_{\prec d} \to \Sigma^r V \to \overline{V} \to 0 \]
we deduce
\begin{align*}
t_1( (\Sigma^r V)_{\prec d} ) &\leqslant \max\{ t_1(\Sigma^r V), \, t_2(\overline{V})\} \\
&\leqslant \max\{ t_0(\Sigma^r W), \, t_2(\overline{V})\} \\
&\leqslant \max\{ r, \, C_{d-1}(r-1) + 3 \} \\
&\leqslant C_{d-1}(r-1) + 3.
\end{align*}
By the induction hypothesis, we have:
\[  \reg( (\Sigma^r V)_{\prec d} ) \leqslant C_{d-1}(C_{d-1}(r-1) + 3).
\]
Hence,
\[ \reg (\Sigma^r V) \leqslant \max\{ \reg( (\Sigma^r V)_{\prec d} ), \, \reg(\overline{V}) \} \leqslant C_{d-1}(C_{d-1}(r-1) + 3 ). \]
Therefore, by Lemma \ref{reg},
\[ \reg(V) \leqslant C_{d-1}(C_{d-1}(r-1) + 3) + r = C_d(r).\]
\end{proof}

\begin{remark} \normalfont
As far as we know, this theorem provides the first explicit upper bound for regularity of $\OI$-modules. We do not know if this bound can be improved substantially.
The careful reader can see that in the proof we have to use $\reg(\overline{V})$ to bound $t_2(\overline{V})$, which significantly amplifies the final upper bound of $\reg(V)$. If a more optimal upper bound for $t_2(\overline{V})$ becomes available as in the case of $\FI$-modules (see the proof of \cite[Theorem 2.4]{li} or $\VI$-modules (see the proof of \cite[Theorem 3.2]{gl3}), then the conclusion of this theorem can be improved.
\end{remark}

By the following corollary, we can do homological algebra safely in the category of $\OI$-modules presented in finite degrees.

\begin{corollary} \label{abelian category}
The category of $\OI$-modules presented in finite degrees is abelian.
\end{corollary}

\begin{proof}
The proof of this result is a routine homological check. Let $\phi: U \to V$ be a morphism in this category. It suffices to show that $\Ker \phi$ and $\coKer \phi$ also lie in it. Breaking this morphism into two short exact sequences
\begin{align*}
& 0 \to \Ker \phi \to U \to \Image \phi \to 0, \\
& 0 \to \Image \phi \to V \to \coKer \phi \to 0,
\end{align*}
one can check that all terms in them are presented in finite degrees.
\end{proof}

For an $\OI$-module $V$ and any $n \in \mathbb{N}$, we define a submodule $\tau_n V$ by letting $(\tau_n V)_i = 0$ for $i <n$ and $(\tau_n V)_i = V_i$ for $i \geqslant n$. The next corollary, which says that $\tau_r V$ has a generalized Koszul property. is an immediate aftermath of \cite[Theorem 5.6]{gl3} and Theorem \ref{regularity theorem}.

\begin{corollary}
If $V$ is presented in finite degrees and $r \geqslant 2^{2^{t_0(V)}} \prd(V)$, then for any $i \in \mathbb{N}$, $H_i(\tau_r V)$ either is 0, or is generated by its value on the object $[r+i]$.
\end{corollary}

\subsection{Inductive machinery}

An important consequence of Corollary \ref{abelian category} is to allow us to extend the inductive machinery introduced in \cite{gl} from the category of finitely generated $\OI$-modules over Noetherian coefficient rings to the category of $\OI$-modules presented in finite degrees. Let us recall \cite[Definition 4]{gl}.

\begin{definition} \label{properties of properties}
Suppose that $\mathscr{T}$ is a subcategory of $\OI \Mod$ and $F: \mathscr{T} \to \mathscr{T}$ is a functor. We say that a property (P) of some $\OI$-modules is:
\begin{itemize}
\item
\emph{glueable on $\mathscr{T}$} if, for every short exact sequence $0 \to U \to V \to W \to 0$ in $\mathscr{T}$:
\begin{equation*}
\mbox{$U$ and $W$ has property (P)} \quad \Longrightarrow \quad \mbox{$V$ has property (P);}
\end{equation*}

\item
\emph{$F$-dominant on $\mathscr{T}$} if, for every $V\in \mathscr{T}$:
\begin{equation*}
\mbox{$FV$ has property (P)} \quad \Longrightarrow \quad \mbox{$V$ has property (P);}
\end{equation*}

\item
\emph{$F$-predominant on $\mathscr{T}$} if, for every $V\in \mathscr{T}$:
\begin{equation*}
\mbox{$FV$ has property (P) and $\kappa V=0$} \quad \Longrightarrow \quad  \mbox{$V$ has property (P).}
\end{equation*}
\end{itemize}
\end{definition}

The following theorem provides a convenient way to check qualitative representation theoretic properties of $\OI$-modules presented in finite degrees.

\begin{theorem}
Let (P) be a property of some $\OI$-modules and suppose that the zero module has property (P). Then every $\OI$-module presented in finite degrees has property (P) if and only if (P) is glueable, $\Sigma$-dominant, and $\Delta$-predominant.
\end{theorem}

\begin{proof}
This theorem actually formalizes the strategy we used to show Theorem \ref{regularity theorem}. One direction is trivial, so we show the other one.

Firstly, since $t_0(\Delta \overline{V}) < t_0(\overline{V}) \leqslant t_0(V)$, the induction hypothesis guarantees $\Delta \overline{V}$ has property (P). Since (P) is $\Delta$-predominant and $\kappa \overline{V} = 0$, $\overline{V}$ has property (P). Similarly, $(\Sigma^rV)_{\prec d}$ has property (P). Since (P) is glueable, by the short exact sequence
\[ 0 \to (\Sigma^rV)_{\prec d} \to \Sigma^r V \to \overline{V} \to 0\]
we conclude that $\Sigma^r V$ has property (P). But (P) is $\Sigma$-predominant, so $V$ has property (P) as well.
\end{proof}

For instances, let (P) be the property of having finite regularity. Then Lemma \ref{reg} asserts that (P) is $\Sigma$-dominant and $\Delta$-predominant. It is easy to see that (P) is glueable. Therefore, by the above theorem we know that every $\OI$ presented in finite degrees has finite regularity.

\subsection{Filtration stability}

If we apply $\Sigma^n$ to an $\OI$-module $V$ presented in finite degrees, it may not become a semi-induced module for $n \gg 0$, as explained in Example \ref{example}. However, we can still get a weaker stability result. That is, there is a finite set of $\OI$-modules such that each $\Sigma^n V$ has a finite filtration whose successive quotients lie in this set.

\begin{theorem}
Let $V$ be an $\OI$-module presented in finite degrees. Then there exist a finite collection of $\OI$-modules $\mathscr{F}_V = \{V^1, \ldots, V^s \}$ and an integer $N \in \N$ such that for every $n \geqslant N$, there is a finite filtration on $\Sigma^n V$ with the property that each successive quotient is isomorphic to a member $V^i \in \mathscr{F}_V$. Moreover, $t_0(V^i) \leqslant t_0(V)$ for $i \in [s]$.
\end{theorem}

\begin{proof}
Let (P) be the property addressed in the theorem. We show that (P) is glueable, $\Sigma$-dominant, and $\Delta$-predominant. But by carefully checking the proof of Theorem \ref{inductive machinery} we find that the glueable condition can be replaced by the following weaker condition:

(w) In the short exact sequence
\[ 0 \to (\Sigma^rV)_{\prec d} \to \Sigma^r V \to \overline{V} \to 0,\]
if the first and the third terms satisfy (P), so does the middle term.
But this is clearly true. Indeed, if $\Sigma^l (\Sigma^rV)_{\prec d}$ and $\Sigma^m \overline{V}$ satisfy (P) for $l \geqslant N_1$ and $m \geqslant N_2$. Then take $N = \max \{N_1, \, N_2 \}$ and let
\[ \mathscr{F}_V = \mathscr{F}_{(\Sigma^rV)_{\prec d}} \cup \mathscr{F}_{\overline{V}}. \]
Note that the generation degree of each module in $\mathscr{F}_{(\Sigma^rV)_{\prec d}}$ (resp., $\mathscr{F}_{\overline{V}}$) is at most $d-1$ (resp., $d$). We conclude that $\Sigma^n V$ satisfies (P) for $n \geqslant N$.

If $\Sigma^l (\Sigma V)$ satisfies (P) for $l \geqslant N$, then clearly $\Sigma^n V$ has property (P) for $n \geqslant N+1$ by letting $\mathscr{F}_V = \mathscr{F}_{\Sigma V}$; that is, (P) is $\Sigma$-dominant.

Now suppose that $\kappa V = 0$, or equivalently the natural map $V \to \Sigma V$ is injective, and $\Sigma^n (\Delta V)$ satisfies property (P) when $n \geqslant N$ for a certain $N \in \N$ and $\mathscr{F}_{\Delta V}$. We apply $\Sigma^n$ to get the exact sequence
\[ 0 \to \Sigma^n V \to \Sigma^{n+1} V \to \Sigma^n (\Delta V) \to 0. \]
We define
\[ \mathscr{F}_V = \mathscr{F}_{\Delta V} \cup \{ \Sigma^N V\}. \]
Note that the generation degree of every member in $\mathscr{F}_V$ is at most $t_0(V)$. Furthermore, $\Sigma^{N+1} V$ has a filtration whose successive quotients all lie in $\mathscr{F}_V$. In the next step, we have
\[ 0 \to \Sigma^{N+1} V \to \Sigma^{N+2} V \to \Sigma^{N+1} (\Delta V) \to 0. \]
Combining the filtration for $\Sigma^{N+1} V$ and the filtration for $\Sigma^{N+1}(\Delta V)$, we deduce that $\Sigma^{N+2} V$ has a filtration whose successive quotients all lie in $\mathscr{F}_V$. By an induction on $n$, we conclude that $\Sigma^n V$ satisfies property $V$ for $n \geqslant N$. That is, the property (P) is $\Delta$-predominant.
\end{proof}

Note that members in the set $\mathscr{F}_V$ might not have ``nice" properties. For instance, consider the module $V$ in Example \ref{example}. A natural choice is $\mathscr{F}_V = \{V, \, U \}$ as specified in the example. One knows that $\kappa V = 0$, but $\kappa U = U$.

One can bound the size of $\mathscr{F}_V$ appearing in the previous statement. That is:

\begin{corollary}
In the previous theorem, one can choose a suitable collection $\mathscr{F}_V$ such that its cardinality does not exceed $2^{d+1} - 1$, where $d = t_0(V)$.
\end{corollary}

\begin{proof}
We use an induction on $d$. By looking at the proof of the previous theorem we find:
\[ \mathscr{F}_V = \mathscr{F}_{\Delta \overline{V}} \cup \{ \Sigma^N \Delta V \} \cup \mathscr{F}_{(\Sigma^rV)_{\prec d}}. \]
Therefore, one obtains:
\[ | \mathscr{F}_V | \leqslant 1+ |\mathscr{F}_{\Delta \overline{V}} | + | \mathscr{F}_{(\Sigma^r V)_{\prec d}} |. \]
But by the induction hypothesis, both $ |\mathscr{F}_{\Delta \overline{V}} |$ and $| \mathscr{F}_{(\Sigma^r V)_{\prec d}} |$ do not exceed $2^d - 1$, so the conclusion holds for $\mathscr{F}_V$.
\end{proof}

\subsection{Hilbert functions}

In this subsection we study the Hilbert function of fintely generated $\OI$-modules $V$ when $\kk$ is a field. It is already know that these functions are eventually polynomial. The following theorem tells us where this phenomenon begins.

\begin{theorem}
Let $V$ be a finitely generated $\OI$-module over a field $k$. Then there exists a rational polynomial $P$ such that $\dim_{\kk} V_n = P(n)$ whenever
\[ n \geqslant 2^{2^{t_0(V)}} \prd(V). \]
Moreover, the degree of $P$ is at most $t_0(V)$.
\end{theorem}

\begin{proof}
The proof is almost the same as that of Theorem \ref{regularity theorem}. Let $d = t_0(V)$ and $r = \prd(V)$. The conclusion holds trivially for $d = -1$ (by convention, we suppose that the degree of the zero polynomial is $-1$), so we suppose that $d \geqslant 0$. It suffices to show the conclusion for $n \geqslant C_d(r)$. We use an induction on $d$.

For $d = 0$, in the proof of Theorem \ref{regularity theorem} we know $\Sigma^r V \cong \overline{V}$. But $\overline{V} \cong \Sigma \overline{V}$ since $\Delta \overline{V} = 0$. This happens if and only if the Hilbert function of $\Sigma^r V$ is a constant function. Equivalently, the conclusion holds for $n \geqslant r = C_0(r)$.

Now suppose that $d \geqslant 1$. By the induction hypothesis, there is a rational polynomial $Q$ with degree at most $d' = t_0((\Sigma^r V)_{\prec d}) \leqslant d-1$ such that
\[ \dim_{\kk} ((\Sigma^r V)_{\prec d})_n = Q(n) \quad \mbox{ for } n \geqslant C_{d'}(r'),\]
where
\[ r' = \max \{t_0((\Sigma^r V)_{\prec d}), \, t_1((\Sigma^r V)_{\prec d}) \} \leqslant \max \{d-1, \, C_{d-1} (r-1) + 3\} = C_{d-1}(r-1) + 3.\]
and the second inequality is shown in the proof of Theorem \ref{regularity theorem}. Therefore,
\[ \dim_{\kk} ((\Sigma^r  V)_{\prec d})_n = Q(n) \quad \mbox{ for } n \geqslant C_{d-1}(C_{d-1}(r-1) + 3).\]

Consider the short exact sequence
\[ 0 \to \overline{V} \to \Sigma \overline{V} \to \Delta \overline{V} \to 0. \]
Note that $t_0(\Delta \overline{V}) \leqslant d-1$ and $t_1 (\Delta \overline{V}) \leqslant r-1$ by Lemma \ref{inductive inequalities}. By an analogue argument, there is a rational polynomial $T$ with degree at most $t_0(\Delta \overline{V}) \leqslant d-1$ such that
\[\dim_{\kk} (\Delta \overline{V})_n = T(n) \quad \mbox{ for } n \geqslant C_{d-1}(r-1).\]
But we know
\[ \dim_{\kk} (\Delta \overline{V})_n = \dim_{\kk} \overline{V}_{n+1} - \dim_{\kk} \overline{V}_n. \]
Consequently, the functions $n \mapsto \dim_{\kk} \overline{V}_n$ coincides with a polynomial with degree at most $d$ for
\[n \geqslant C_{d-1}(r-1) + 1. \]
By the short exact sequence
\[ 0 \to (\Sigma^r  V)_{\prec d} \to \Sigma^r V \to \overline{V} \to 0,\]
we know that the function $n \mapsto \dim_{\kk} (\Sigma^r  V)_n$ is a polynomial with degree at most $d$ for $n \geqslant C_{d-1}(C_{d-1}(r-1) + 3)$. This is equivalent to saying that the function
\[ n \mapsto \dim_{\kk} V_n, \quad n \geqslant C_{d-1}(C_{d-1}(r-1) + 3) + r = C_d(r) \]
coincides with a rational polynomial with degree at most $d$.
\end{proof}

\subsection{Semi-induced modules}

In this subsection we describe several homological characterizations of semi-induced modules. These modules have been defined in \cite[Subsection 5.6]{gl2} in the name of relative projective modules. Let us make an explicit description here.

Let $\OB$ be the subcategory of $\OI$ whose objects coincide, but morphisms in $\OB$ are bijections. That is, morphisms in $\OB$ are of the form $\mathrm{id}: [n] \to [n]$, $n \in \mathbb{N}$. The embedding functor $\epsilon: \OB \to \OI$ has a left adjoint functor $\rho: \OI \to \OB$, which induces a functor $\rho^{\ast}: \OB \Mod \to \OI \Mod$. We call $\OI$-modules isomorphic to $\rho^{\ast} T$ \textit{induced modules}, where $T$ is an $\OB$-module such that $T_n = 0$ for $n \gg 0$. An $\OI$-module is \textit{semi-induced} if it has a finite filtration such that each subquotient is an induced module. Clearly, semi-induced modules are presented in finite degrees, and include projective modules generated in finite degrees as special examples. In particular, when $\kk$ is a field, semi-induced modules are projective. We also remind the reader that the category of $\OI$-modules presented in finite degrees is abelian, so we can do homological algebra safely.

\begin{proposition} \cite[Proposition 5.3]{gl2} \label{part 1}
Let $V$ be an $\OI$-module presented in finite degrees. Then the following statements are equivalent:
\begin{enumerate}
\item $V$ is semi-induced;
\item $\h^{\OI}_1(V) = 0$;
\item $\h^{\OI}_i(V) = 0$ for all $i \geqslant 1$.
\end{enumerate}
\end{proposition}

\begin{proof}
The argument in the proof of \cite[Proposition 5.3]{gl2} still works for modules presented in finite degrees.
\end{proof}

Let $V$ be a semi-induced $\OI$-module with $t_0(V) = n$. By the following lemma, it has a natural filtration of induced modules as follows: $0 \subseteq V_{\prec 1} \subseteq \ldots \subseteq V_{\prec n} \subseteq V_{\prec n+1} = V$.

\begin{lemma} \label{natural filtration}
Let $V$ be as above. Then for $i \in [n+1]$, $V_{\prec i}$ and $V/V_{\prec i}$ are semi-induced modules, and $V_{\prec i} / V_{\prec i-1}$ is an induced module.
\end{lemma}

\begin{proof}
We firstly show that $V_{\prec i}$ is a semi-induced module. The short exact sequence
\begin{equation*}
0 \to V_{\prec i} \to V \to V/V_{\prec i} \to 0
\end{equation*}
induces an exact sequence
\begin{equation*}
0 \to \h_1^{\OI}(V/V_{\prec i}) \to \h_0^{\OI}(V_{\prec i}) \to \h_0^{\OI}(V) \to \h_0^{\OI}(V/V_{\prec i}) \to 0
\end{equation*}
since $\h_1^{\OI}(V) = 0$ by the previous proposition. Now note that as $\OI$-modules, $(\h_0^{\OI}(V_{\prec i})_s = 0$ for $s \geqslant i$, while $(\h_1^{\OI}(V/V_{\prec i}))_s = 0$ for $s < i$. Therefore, the only possibility is $\h_1^{\OI}(V/V_{\prec i}) = 0$; that is, $V/V_{\prec i}$ is semi-induced. But the exact sequence
\begin{equation*}
0 = \h_2^{\OI}(V/V_{\prec i}) \to \h_1^{\OI}(V_{\prec i}) \to \h_1^{\OI}(V) = 0
\end{equation*}
tells us that $\h_1^{\OI}(V_{\prec i}) = 0$, so $V_{\prec i}$ is semi-induced as well.

Clearly, $V_{\prec i} / V_{\prec i-1}$ is generated by its value on the object $[i-1]$. The above arguments also tells us that it is semi-induced via replacing $V$ and $V_{\prec i}$ by $V_{\prec i}$ and $V_{\prec i-1}$ respectively. Therefore, it is an induced module.
\end{proof}

We collect more properties of semi-induced modules in the following lemma.

\begin{lemma} \label{properties of semi-induced modules}
Let $V$ be an $\OI$-module presented in finite degrees. Then:
\begin{enumerate}
\item if $V$ is semi-induced, then the natural map $V \to \Sigma V$ is injective;
\item if $V$ is semi-induced, so are $\Sigma V$ and $\Delta V$;
\item if the natural map $V \to \Sigma V$ is injective, and $\Delta V$ is semi-induced, then $V$ is semi-induced as well; that is, being semi-induced is $\Delta$-predominant.
\end{enumerate}
\end{lemma}

\begin{proof}
(1) Since $V$ has a natural filtration by Lemma \ref{natural filtration}, it suffices to consider the case that $V = \rho^{\ast}(T)$, where $T$ is an $\OB$-module supported on a certain object $[n]$. But by the combinatorial structure of $\OI$ and the definition of $\Sigma$, one can check that there is a short exact sequence
\begin{equation*}
0 \to \rho^{\ast} (T) \to \Sigma \rho^{\ast}(T) \to \rho^{\ast} (T') \to 0
\end{equation*}
where $T'$ is isomorphic to $T$ as $\kk$-modules, and is viewed as an $\OB$-module supported on the object $[n-1]$ (in the special case $n = 0$, we let $T' = 0$). In particular, this implies (1).

(2) If $0 \to U \to V \to W \to 0$ is a short exact sequence such that the natural map $W \to \Sigma W$ is injective, then applying the snake lemma to the commutative diagram
\begin{equation*}
\xymatrix{
0 \ar[r] & U \ar[r] \ar[d] & V \ar[r] \ar[d] & W \ar[r] \ar[d] & 0\\
0 \ar[r] & \Sigma U \ar[r] & \Sigma V \ar[r] & \Sigma W \ar[r] & 0
}
\end{equation*}
one deduces a short exact sequence
\begin{equation*}
0 \to \Delta U \to \Delta V \to \Delta W \to 0.
\end{equation*}
Combining this fact and the conclusion of Lemma \ref{natural filtration}, we know that the natural filtration
\begin{equation*}
0 = V_{\prec 0} \subseteq V_{\prec 1} \subseteq \ldots \subseteq V_{\prec n+1} = V
\end{equation*}
of $V$ with $n = t_0(V)$ gives a filtration
\begin{equation*}
0 = \Delta V_{\prec 0} \subseteq \Delta V_{\prec 1} \subseteq \ldots \subseteq \Delta V_{\prec n+1} = \Delta V.
\end{equation*}
Therefore, to show that $\Delta V$ is semi-induced, it suffices to show that
\begin{equation*}
\Delta V_{\prec i}/\Delta V_{\prec i-1} \cong \Delta (V_{\prec i} / V_{\prec i-1})
\end{equation*}
is an induced module for each $i \in [n+1]$. But this is clear from the proof of (1) since
\begin{equation*}
V_{\prec i} / V_{\prec i-1} \cong \rho^{\ast} ((V_{\prec i} / V_{\prec i-1})_{i-1}).
\end{equation*}

We have shown that $\Delta V$ is semi-induced. Clearly, $\Sigma V$ is semi-induced as well by the short exact sequence $0 \to V \to \Sigma V \to \Delta V \to 0$.

(3) First we handle the case that $V$ is generated by its value $V_n$ on a certain object $[n]$. If $n = 0$, then $\Delta V = 0$, so $V \cong \Sigma V$. In this case, one knows that $V \cong \rho^{\ast}(V_0)$ is induced. If $n > 0$, then by the short exact sequence
\begin{equation*}
0 \to W \to \rho^{\ast}(V_n) \to V \to 0
\end{equation*}
and the assumption that the map $V \to \Sigma V$ is injective, we obtain a short exact sequence
\begin{equation*}
0 \to \Delta W \to \Delta \rho^{\ast}(V_n) \to \Delta V \to 0.
\end{equation*}
By the given condition, $\Delta V$ is an induced module generated by its value on the object $[n-1]$, and $\Delta \rho^{\ast}(V_n)$ is also an induced module generated by its value on the object $[n-1]$. Since their values on $[n-1]$ coincides, we conclude that $\Delta W = 0$, which forces $W = 0$ since otherwise
\begin{equation*}
t_0(\Delta W) = t_0(W) - 1 > n - 1 \geqslant 0
\end{equation*}
Consequently, $V \cong \rho^{\ast}(V_n)$ is induced.

Now we consider the general case. We use an induction on $t_0(V) = n$. The conclusion holds for $n = 0$. Suppose that $n \geqslant 1$. Let
\begin{equation*}
0 \to W \to P \to V \to 0
\end{equation*}
be a short exact sequence such that $t_0(P) = t_0(V)$. It induces a short exact sequence
\begin{equation*}
0 \to \Delta W \to \Delta P \to \Delta V \to 0
\end{equation*}
since the natural map $V \to \Sigma V$ is injective. Therefore, we have
\begin{equation*}
t_1(V) \leqslant t_0(W) \leqslant t_0(\Delta W) + 1 \leqslant \max \{ t_0(\Delta P) + 1, \, t_1(\Delta V) + 1 \} = t_0(\Delta P) + 1 = t_0(V)
\end{equation*}
since $\h^{\OI}_1(\Delta V) = 0$ and $t_0(\Delta P) = t_0(P) - 1 = t_0(V) -1$. Thus from the short exact sequence
\begin{equation}
0 \to V_{\prec n} \to V \to V/V_{\prec n} \to 0
\end{equation}
we deduce
\begin{equation*}
t_1(V/V_{\prec n}) \leqslant \max \{t_0(V_{\prec n}), \, t_1(V) \} \leqslant \max \{n-1, \, t_0(V)\} = n.
\end{equation*}
By looking at the short exact sequence
\begin{equation*}
0 \to M \to \rho^{\ast} ((V/V_{\prec n})_n) \to V/V_{\prec n} \to 0,
\end{equation*}
we conclude that $t_1(V/V_{\prec n}) \leqslant n$ happens if and only if $M = 0$ since otherwise $M$ is supported on objects $[m]$ with $m > n$; that is, $V/V_{\prec n}$ is an induced module. Therefore, the natural map $V/V_{\prec n} \to \Sigma (V/V_{\prec n})$ is injective, and we get a short exact sequence
\begin{equation*}
0 \to \Delta V_{\prec n} \to \Delta V \to \Delta (V/V_{\prec n}) \to 0.
\end{equation*}
Since the last two terms in it are semi-induced, so is the first term by considering the long exact sequence of homology groups. But the natural map $V_{\prec n} \to \Sigma V_{\prec n}$ is also injective. By the induction hypothesis, $V_{\prec n}$ is semi-induced. Consequently, $V$ is semi-induced as well.
\end{proof}

Now we are ready to prove Theorem \ref{homologically acyclic modules}, which is a combination of Proposition \ref{part 1} and the following proposition.

\begin{proposition}
Let $0 \to U \to V \to W \to 0$ be a short exact sequence of $\OI$-modules presented in finite degrees. If two terms are semi-induced, so is the third one. In particular, if $\h^{\ast}_i(V) = 0$ for a certain $i \in \mathbb{N}$, then $V$ is semi-induced.
\end{proposition}

\begin{proof}
If $U$ and $W$ are semi-induced, clearly $V$ is semi-induced as well. If $V$ and $W$ are semi-induced, then by looking at the long exact sequence of homology groups we observe that $\h^{\OI}_1(U) = 0$, so $U$ is semi-induced. Now assume that both $U$ and $V$ are semi-induced, and apply an induction on the invariant $n = \max \{t_0(U), \, t_0(V) \}$. The conclusion holds trivially for $n = -1$, so we assume that $n \geqslant 0$.

Note that the commutative diagram
\begin{equation*}
\xymatrix{
0 \ar[r] & U \ar[r] \ar[d] & V \ar[r] \ar[d] & W \ar[r] \ar_{\delta}[d] & 0\\
0 \ar[r] & \Sigma U \ar[r] & \Sigma V \ar[r] & \Sigma W \ar[r] & 0
}
\end{equation*}
gives rise to an exact sequence
\begin{equation*}
0 \to \Ker \delta \to \Delta U \to \Delta V \to \Delta W \to 0,
\end{equation*}
so we can identify $\Ker \delta$ with a submodule of $\Delta U$. But since $U$ is semi-induced, $\Delta U$ is semi-induced as well. Therefore, for any nonzero submodule $U' \subseteq \Delta U$, the natural map $U' \to \Sigma U'$ is injective, and hence nonzero. But the natural map $\Ker \delta \to \Sigma \Ker \delta$ is 0. This forces $\Ker \delta$ to be 0. Thus we obtain a short exact sequence
\begin{equation*}
0 \to \Delta U \to \Delta V \to \Delta W \to 0.
\end{equation*}
Note that the first two terms in this sequence are semi-induced, and
\begin{equation*}
\max \{t_0(\Delta U), \, t_0(\Delta V) \} < \max \{t_0(U), \, t_0(V) \}.
\end{equation*}
By the induction hypothesis, $\Delta W$ is semi-induced. By (3) of Lemma \ref{properties of semi-induced modules}, $W$ is semi-induced as well.

Now suppose that $\h^{\OI}_i(V) = 0$. If $i = 0$, then $V = 0$; if $i = 1$, then $V$ is semi-induced by Proposition \ref{part 1}. Suppose that $i > 1$. Consider a short exact sequence
\begin{equation*}
0 \to W \to P \to V \to 0.
\end{equation*}
Then $\h^{\OI}_{i-1} (W) = 0$. By the induction hypothesis, $W$ is semi-induced, so is $V$.
\end{proof}

With careful modifications, main results and their proofs in \cite[Subsection 4.2]{ly} for finitely generated $\FI$-modules over Noetherian coefficient rings actually hold for $\OI$-modules presented in finite degrees. In particular, we classify all $\OI$-modules presented in finite degrees whose projective dimension is finite.

\begin{theorem}
Let $V$ be an $\OI$-module presented in finite degrees with $t_0(V) = n$. Then the projective dimension $\pd_{\OI} (V)$ is finite if and only if the following conditions hold:
\begin{enumerate}
\item $V$ is semi-induced;
\item for $i \in [n+1]$, $\pd_{\kk} ((V_{\prec i} / V_{\prec i-1})_{i-1}) < \infty$.
\end{enumerate}
In this case, one has
\begin{equation*}
\pd_{\OI} (V) = \max \{ \pd_{\kk} ((V_{\prec i} / V_{\prec i-1})_{i-1}) \mid i \in [n+1] \}.
\end{equation*}
\end{theorem}

In particular, if $\kk$ has global dimension 0, then $V$ has finite projective dimension if and only if it is projective.

\section{Further questions}

There are still quite a lot of questions to be answered for a satisfactory understanding on the complete picture of $\OI$-modules. Here we list a few question, which we believe deserve further research.

\begin{question}
Develop a torsion theory and a local cohomology theory in the category of $\OI$-modules presented in finite degrees, as the second author and Ramos did for $\FI$-modules in \cite{lr}.
\end{question}

This question has been answered for $\FI$-modules via using the following crucial fact: for $\FI$-modules, the functors $\Sigma$ and $\Delta$ commutes; that is, $\Sigma \circ \Delta \cong \Delta \circ \Sigma$. Unfortunately, this does not hold any longer for $\OI$-modules. We also remark that when $\kk$ is a field,  G\"{u}nt\"{u}rk\"{u}n and Snowden provided answers for graded modules of the increasing monoid.

As Church, Miller, Nagpal and Reinhold did in \cite{cmnr} for $\FI$-modules, for an $\OI$-module $V$ presented in finite degrees, we define its stable degree to be
\[ \std(V) = \min \{ t_0(\Sigma^n V) \mid n \in \N \}. \]
Since $t_0(\Sigma V) \leqslant t_0(V) < \infty$, we know that $\std(V)$ is finite. Furthermore, there exists a number $N \in \N$ such that
\[ t_0(\Sigma^N V) = t_0(\Sigma^{N+1} V) = \ldots. \]
Our next question is:

\begin{question}
Let $V$ be an $\OI$-module presented in finite degrees. Describe an upper bound for $N$ in terms of $t_0(V)$ and $t_1(V)$ such that $t_0(\Sigma^n V) = \std(V)$ for $n \geqslant N$.
\end{question}

Theorem \ref{filtration stability} also raises a few interesting questions.

\begin{question}
Let $V$ be an $\OI$-module presented in finite degrees and use the notation in Theorem \ref{filtration stability}.
\begin{enumerate}
\item Describe the $\OI$-modules in the finite set $\mathscr{F}_V$ (although $\mathscr{F}_V$ might not be unique).
\item Describe an upper bound for $N$ such that for $n \geqslant N$, $\Sigma^n V$ has a filtration for which all successive quotients lie in $\mathscr{F}_V$.
\item Describe the asymptotic behavior of multiplicities $c_i$ such that a filtration of $\Sigma^n V$ for $n \gg 0$ contains exactly $c_i$ copies of successive quotients isomorphic to $V_i \in \mathscr{F}_V$.
\end{enumerate}
\end{question}

\end{document}